\documentclass{amsart}

\usepackage{amsmath}
\usepackage{amsthm}
\usepackage{amssymb}
\usepackage{amsfonts,mathrsfs}
\usepackage{stmaryrd}
\usepackage{cleveref}

\theoremstyle{plain}
\newtheorem{thm}{Theorem}

\newtheorem{lem}{Lemma}

\theoremstyle{definition}

\newtheorem{exm}{Example}

\theoremstyle{remark}

\newcommand{\N}{\mathbb{N}}

\newcommand{\R}{\mathbb{R}}
\newcommand{\C}{\mathbb{C}}
\newcommand{\D}{\mathbb{D}}
\newcommand{\B}{\mathbb{B}}

\newcommand{\coloneqq}{\mathrel{\mathop:}=}

\providecommand{\abs}[1]{\lvert#1\rvert}
\providecommand{\gabs}[1]{\big\lvert#1\big\rvert}

\DeclareMathOperator{\SUPP}{supp}
\DeclareMathOperator{\DIST}{dist}

\DeclareMathOperator{\RE}{Re}

\providecommand{\zbar}{{\bar{z}}}

\makeatletter
\def\moverlay{\mathpalette\mov@rlay}
\def\mov@rlay#1#2{\leavevmode\vtop{%
   \baselineskip\z@skip \lineskiplimit-\maxdimen
   \ialign{\hfil$\m@th#1##$\hfil\cr#2\crcr}}}
\newcommand{\charfusion}[3][\mathord]{
    #1{\ifx#1\mathop\vphantom{#2}\fi
        \mathpalette\mov@rlay{#2\cr#3}
      }
    \ifx#1\mathop\expandafter\displaylimits\fi}
\makeatother

\makeatletter
\let\save@mathaccent\mathaccent
\newcommand*\if@single[3]{%
  \setbox0\hbox{${\mathaccent"0362{#1}}^H$}%
  \setbox2\hbox{${\mathaccent"0362{\kern0pt#1}}^H$}%
  \ifdim\ht0=\ht2 #3\else #2\fi
  }
\newcommand*\rel@kern[1]{\kern#1\dimexpr\macc@kerna}
\newcommand*\widebar[1]{\@ifnextchar^{{\wide@bar{#1}{0}}}{\wide@bar{#1}{1}}}
\newcommand*\wide@bar[2]{\if@single{#1}{\wide@bar@{#1}{#2}{1}}{\wide@bar@{#1}{#2}{2}}}
\newcommand*\wide@bar@[3]{%
  \begingroup
  \def\mathaccent##1##2{%
    \let\mathaccent\save@mathaccent
    \if#32 \let\macc@nucleus\first@char \fi
    \setbox\z@\hbox{$\macc@style{\macc@nucleus}_{}$}%
    \setbox\tw@\hbox{$\macc@style{\macc@nucleus}{}_{}$}%
    \dimen@\wd\tw@
    \advance\dimen@-\wd\z@
    \divide\dimen@ 3
    \@tempdima\wd\tw@
    \advance\@tempdima-\scriptspace
    \divide\@tempdima 10
    \advance\dimen@-\@tempdima
    \ifdim\dimen@>\z@ \dimen@0pt\fi
    \rel@kern{0.6}\kern-\dimen@
    \if#31
      \overline{\rel@kern{-0.6}\kern\dimen@\macc@nucleus\rel@kern{0.4}\kern\dimen@}%
      \advance\dimen@0.4\dimexpr\macc@kerna
      \let\final@kern#2%
      \ifdim\dimen@<\z@ \let\final@kern1\fi
      \if\final@kern1 \kern-\dimen@\fi
    \else
      \overline{\rel@kern{-0.6}\kern\dimen@#1}%
    \fi
  }%
  \macc@depth\@ne
  \let\math@bgroup\@empty \let\math@egroup\macc@set@skewchar
  \mathsurround\z@ \frozen@everymath{\mathgroup\macc@group\relax}%
  \macc@set@skewchar\relax
  \let\mathaccentV\macc@nested@a
  \if#31
    \macc@nested@a\relax111{#1}%
  \else
    \def\gobble@till@marker##1\endmarker{}%
    \futurelet\first@char\gobble@till@marker#1\endmarker
    \ifcat\noexpand\first@char A\else
      \def\first@char{}%
    \fi
    \macc@nested@a\relax111{\first@char}%
  \fi
  \endgroup
}
\makeatother

\subjclass[2010]{32T99, 32U99}
\begin{document}

\title[Smoothing of psh functions on unbounded domains]{On smoothing of plurisubharmonic functions on unbounded domains}
\author{T. Harz}
\begin{abstract}
We prove that for every $n \ge 2$, there exists a pseudoconvex domain $\Omega \subset \C^n$ such that $\mathfrak{c}^0(\Omega) \subsetneq \mathfrak{c}^1(\Omega)$, where $\mathfrak{c}^k(\Omega)$ denotes the core of $\Omega$ with respect to $\mathcal{C}^k$-smooth plurisubharmonic functions on $\Omega$. Moreover, we show that there exists a bounded continuous plurisubharmonic function on $\Omega$ that is not the pointwise limit of a sequence of $\mathcal{C}^1$-smooth bounded plurisubharmonic functions on $\Omega$.
\end{abstract}
\maketitle
\section{Introduction}
Let $\Omega \subset \C^n$ be a domain. If $\Omega$ is bounded, then there always exists a strictly plurisubharmonic function $\varphi$ on $\Omega$ that is bounded from above; for example, $\varphi(z) \coloneqq |\cdot|^2$, where $|\cdot|$ denotes the Euclidean norm on $\C^n$, is a function as desired. However, if $\Omega$ is unbounded, then in general there does not exist a plurisubharmonic function on $\Omega$ that is bounded from above and strictly plurisubharmonic on the whole of $\Omega$. The subset of $\Omega$ where all bounded from above plurisubharmonic functions fail to be strictly plurisubharmonic is called the core of $\Omega$, and will be denoted by $\mathfrak{c}(\Omega)$. 

Core sets have been introduced and studied in the series of articles \cite{HarzShcherbinaTomassini17}, \cite{HarzShcherbinaTomassini20}, \cite{HarzShcherbinaTomassini21}. Further results related to cores of complex manifolds were obtained, for example, in \cite{GallagherHarzHerbort17}, \cite{PoletskyShcherbina19}, \cite{Slodkowski19}. To be more precise, core sets have been studied with respect to different subclasses of (bounded from above) plurisubharmonic functions on $\Omega$: In \cite{HarzShcherbinaTomassini17} and \cite{HarzShcherbinaTomassini20}, the authors consider cores with respect to the class of $\mathcal{C}^\infty$-smooth plurisubharmonic functions. In \cite{HarzShcherbinaTomassini21}, also core sets related to $\mathcal{C}^k$-smooth plurisubharmonic functions, $k \ge 0$, are investigated. The papers \cite{PoletskyShcherbina19} and \cite{Slodkowski19} deal with cores related to continuous plurisubharmonic functions, and in \cite{GallagherHarzHerbort17} also a core set with respect to a class of plurisubharmonic functions with weak singularities is introduced.

The goal of this paper is to start an investigation on how these different notions of core sets are related. More precisely, we are interested in the question if core sets with respect to certain subclasses of (bounded from above) plurisubharmonic functions do always coincide, or can be different from one another in general. We will show, in particular, that the cores with respect to continuous and with respect to $\mathcal{C}^1$-smooth plurisubharmonic functions are not always the same, even on pseudoconvex domains.

Before stating the results of this paper in detail, we need to introduce some notations: We write $\mathcal{PSH}(\Omega)$ to denote the set of plurisubharmonic functions on $\Omega$. Moreover, $\mathcal{PSH}^\ast(\Omega) \coloneqq \mathcal{PSH}(\Omega) \cap L^\infty_{loc}(\Omega)$, where $L^\infty_{loc}(\Omega)$ denotes the set of locally bounded functions on $\Omega$, and $\mathcal{PSH}^k(\Omega) \coloneqq \mathcal{PSH}(\Omega) \cap \mathcal{C}^k(\Omega)$ for every $k \in \N_0^\infty \coloneqq \{0, 1, 2, \ldots\} \cup {\infty}$. A function $\varphi \colon \Omega \to [-\infty, \infty)$, $\varphi \not\equiv -\infty$, is called strictly plurisubharmonic if for every $\mathcal{C}^\infty$-smooth function $\theta \colon \Omega \to \R$ with compact support there exists a constant $\varepsilon_0 > 0$ such that for all $\varepsilon \in (-\varepsilon_0, \varepsilon_0)$ the function $\varphi + \varepsilon\theta$ is plurisubharmonic on $\Omega$ (note that in the case that $\varphi$ is $\mathcal{C}^2$-smooth, this definition coincides with the usual definition of strict plurisubharmonicity by means of positive definiteness of the complex Hessian of $\varphi$). We say that $\varphi$ is strictly plurisubharmonic near a point $z \in \Omega$, if $\varphi$ is strictly plurisubharmonic on some open neighborhood $U \subset \Omega$ of $z$. 

The set $\mathfrak{c}(\Omega)$ related to the function space $\mathcal{PSH}(\Omega)$ is then defined by
\begin{equation*}
\mathfrak{c}(\Omega) \coloneqq \left\{z \in \Omega : \begin{minipage}{.61\textwidth} \center every function $\varphi \in \mathcal{PSH}(\Omega)$ that is bounded from \\ above fails to be strictly plurisubharmonic near $z$  \end{minipage} \right\}
\end{equation*}
In the same way, we introduce the sets $\mathfrak{c}^\ast(\Omega)$ and $\mathfrak{c}^k(\Omega)$ related to the function spaces $\mathcal{PSH}^\ast(\Omega)$ and $\mathcal{PSH}^k(\Omega)$, $k \in \N_0^\infty$, respectively.
It follows immediately from the definition that we have a chain of inclusions as follows,
$$ \mathfrak{c}(\Omega) \subset \mathfrak{c}^\ast(\Omega) \subset \mathfrak{c}^0(\Omega) \subset \mathfrak{c}^1(\Omega) \subset \cdots .$$
%
We will show that each of the first three inclusions above can be strict in general. In fact, to see that in some cases $\mathfrak{c}(\Omega) \neq \mathfrak{c}^\ast(\Omega)$ is easy, and it can be observed, e.g., in the following example. 

\begin{exm}\label{ex:coreline}
For generic $C \in \R$, the set
$$ \Omega \coloneqq \big\{(z,w) \in \C^2: \log\abs{w} + \abs{z}^2 + \abs{w}^2 < C \big\} $$
is a strictly pseudoconvex domain with smooth boundary in $\C^2$. Clearly, $\mathfrak{c}(\Omega) = \varnothing$, since $\varphi(z,w) \coloneqq \log\abs{w} + (\abs{z}^2 + \abs{w}^2) - C$ is strictly plurisubharmonic on $\Omega$. On the other hand, observe that $\Omega$ contains the complex line $L \coloneqq \C \times \{0\}$. Hence if $\varphi \in \mathcal{PSH}^\ast(\Omega)$ is bounded from above, then, by Liouville's theorem, there exists a constant $M \in \R$ such that $\varphi \equiv M$ on $L$. In particular, $\varphi$ is not strictly plurisubharmonic at any point in $L$. Thus $L \subset \mathfrak{c}^\ast(\Omega)$. In particular, $\mathfrak{c}(\Omega) \neq \mathfrak{c}^\ast(\Omega)$.
\end{exm}

To see that in general $\mathfrak{c}^\ast(\Omega) \neq \mathfrak{c}^0(\Omega)$ and $\mathfrak{c}^0(\Omega) \neq \mathfrak{c}^1(\Omega)$, even if $\Omega$ is pseudoconvex, is more difficult. This is shown in the following two results of this article.

%
%

\begin{thm} \label{thm:locboundedC0}
For every $n \ge 2$, there exist a pseudoconvex domain $\Omega \subset \C^n$ and a nonempty open set $\omega \subset \Omega$ such that the following assertions hold true:
\begin{enumerate}
 \item There exists a function $\varphi \in \mathcal{PSH}(\Omega) \cap L^\infty_{loc}(\Omega)$ such that $\varphi < 0$ on $\Omega$ and $\varphi$ is strictly plurisubharmonic on $\omega$.
 \item Every function $\varphi \in \mathcal{PSH}(\Omega) \cap \mathcal{C}^0(\Omega)$ such that $\varphi < 0$ on $\Omega$ is constant on $\omega$.
\end{enumerate}
In particular, $\mathfrak{c}^\ast(\Omega) \neq \mathfrak{c}^0(\Omega)$.
\end{thm}

\begin{thm} \label{thm:C0C1}
For every $n \ge 2$, there exist a pseudoconvex domain $\Omega \subset \C^n$ and a nonempty open set $\omega \subset \Omega$ such that the following assertions hold true:
\begin{enumerate}
 \item There exists a function $\varphi \in \mathcal{PSH}(\Omega) \cap \mathcal{C}^0(\Omega)$ such that $\varphi < 0$ on $\Omega$ and $\varphi$ is strictly plurisubharmonic on $\omega$.
 \item Every function $\varphi \in \mathcal{PSH}(\Omega) \cap \mathcal{C}^1(\Omega)$ such that $\varphi < 0$ on $\Omega$ is constant on $\omega$.
\end{enumerate}
In particular, $\mathfrak{c}^0(\Omega) \neq \mathfrak{c}^1(\Omega)$.
\end{thm}

The proofs of Theorem 1 and Theorem 2 are given in sections 2 and 3, respectively. In section 4 we note a consequence with respect to non-approximability of bounded plurisubharmonic functions on unbounded domains.

\section{Proof of \Cref{thm:locboundedC0}}
In what follows, $n \ge 2$ is fixed. If $(z_0,w_0) \in \C \times \C^{n-1}$ and $r > 0$, then $\D(z_0,r) \coloneqq \{z \in \C : \abs{z-z_0} < r\}$ and $\B(w_0,r) \coloneqq \{w \in \C^{n-1} : \abs{w-w_0} < r\}$. Moreover, $\D \coloneqq \D(0,1)$ and $\B \coloneqq \B(0,1)$. 

Let $\{\theta_j\}_{j=1}^\infty$ be a dense subset of $[0, 2\pi]$ and let $a_j \coloneqq (1+\frac{1}{j})e^{i\theta_j}$. Define
$$ 
\sigma(z) \coloneqq \sum_{j=1}^\infty \delta_j \log\abs{z-a_j}, 
$$
where $(\delta_j)_{j=1}^\infty$ is a sequence of constants $\delta_j > 0$ converging to zero so fast that
\begin{itemize}
  \item[(1)] $\sigma \colon \C \to [-\infty, \infty)$ is subharmonic,
  \item[(2)] $\abs{\sigma} < 1$ on $\bar{\D}$.
\end{itemize}
Further, fix a point $w_0 \in \C^{n-1}$ such that $\abs{w_0} = 2$, and let
$$ 
\Omega \coloneqq \big\{(z,w) \in \C \times \C^{n-1} : \sigma(z) + \log\abs{z} + \textstyle\frac{1}{2}\log\abs{w-w_0} + \abs{z}^2 + \abs{w}^2 < 4 \big\}. 
$$
Observe that, after possibly passing to a suitable connected component, $\Omega$ is a pseudoconvex domain such that
\begin{enumerate}
  \item[(3)] $(A \cup \{0\}) \times \C^{n-1} \subset \Omega$, where $A \coloneqq \{a_j\}_{j=1}^\infty$,
  \item[(4)] $\C \times \{w_0\} \subset \Omega$,
  \item[(5)] $\bar{\D} \times \bar{\B} \subset \Omega$;
\end{enumerate}
in order to see (5), note that $\log\abs{w-w_0} < \log 3 < 2$ if $\abs{w} < 1$. Define $\tilde{\varphi} \colon \C^n \to [-\infty, \infty)$ as
$$ \tilde{\varphi}(z,w) \coloneqq \sigma(z) + \log\abs{z} + \textstyle\frac{1}{2}\log\abs{w-w_0} + \abs{z}^2 + \frac{1}{2}\abs{w}^2. $$
Then observe that
\begin{enumerate}
  \item[(6)] $\tilde{\varphi}$ is strictly plurisubharmonic,
  \item[(7)] $\tilde{\varphi} > -2$ on $\{\frac{1}{2} < \abs{z} < 1\} \times \B$, and
  \item[(8)] $\tilde{\varphi} < -2$ on $\Omega \cap \{\abs{w} > 4\}$;
\end{enumerate}
in order to see (8), note that $\tilde\varphi <  4 - \frac{1}{2} \abs{w}^2$ on $\Omega$. Thus $\varphi \coloneqq \max(\tilde{\varphi}, -2)$ is a bounded plurisubharmonic function on $\Omega$ such that, according to (6)  and (7), $\varphi$ is strictly plurisubharmonic on 
$$ \omega \coloneqq \{\textstyle\frac{1}{2} < \abs{z} < 1\} \times \B \subset \Omega. $$
In particular, $\omega \cap \mathfrak{c}^\ast(\Omega) = \varnothing$. 

On the other hand, let now $\psi \colon \Omega \to (-\infty, 0)$ be a continuous plurisubharmonic function. Then, by (3) and (4), $\psi \equiv C$ on $(A \cup \{0\}) \times \C^{n-1}$ for some constant $C \in \R$. Thus, by (5), by continuity of $\psi$, and since $b\D \subset \bar{A}$, we conclude that $\psi \equiv C$ on $(b\D \cup \{0\}) \times \B$. An application of the maximum principle to the functions $\psi(\,\cdot\,,w)|_{\D}$, $w \in \B$, thus implies that $\psi \equiv C$ on $\D \times \B$. In particular, $\psi$ fails to be strictly plurisubharmonic on $\omega \subset \D \times \B$, and hence $\omega \subset \mathfrak{c}^0(\Omega)$.

\section{Proof of \Cref{thm:C0C1}}
The proof of \Cref{thm:C0C1} relies on the following simple lemma in dimension 1. Here, for $U \subset \C$ open we write $\mathcal{SH}(U)$ for the set of subharmonic functions on $U$.

\begin{lem} \label{thm:shdimension1} Fix a dense subset $\{\theta_j\}_{j=1}^\infty \subset [0,2\pi]$ and set $a_j \coloneqq (1+\frac{1}{j})e^{i\theta_j}$. Then the following assertions hold true:
\begin{enumerate}
 \item There exists a function $u \in \mathcal{SH}(\C) \cap \mathcal{C}^0(\C)$ such that $u(a_j) = 1$ for every $j \in \N$, and $u(z) = \abs{z}^2$ on $\bar{\D}$.
 \item For every function $v \in \mathcal{SH}(U) \cap \mathcal{C}^1(U)$, where $U \subset \C$ is an open neighborhood of $\bar{\D}$, the following statement is true: if $v(a_j) = 1$ for every $j \in \N$ such that $a_j \in U$, then $v \equiv 1$ on $\bar{\D}$.
\end{enumerate}
\end{lem}
\begin{proof}
(1) Fix a sequence $(r_j)$ of constants $r_j \in (0,1)$ such that $\bar{\D}(a_j,r_j) \cap \bar{\D} = \emptyset$ for every $j \in \N$, and $\bar{\D}(a_j,r_j) \cap \bar{\D}(a_k,r_k) = \emptyset$ for every $j,k \in \N$ such that $j \neq k$. Further, choose a smooth function $\chi \colon \C \to [0,\infty)$ such that $\SUPP \chi \subset \D$ and $\chi \equiv 1$ near $0$. Finally, let $(\varepsilon_j)$ be a sequence of constants $\varepsilon_j > 0$ such that for every $j \in \N$ the function $\abs{z}^2 + \varepsilon_j\chi(\abs{z-a_j}/r_j)\log\abs{z-a_j}$ is strictly subharmonic on $\D(a_j,r_j)$.\medskip
 
Define functions $u_m \in \mathcal{SH}(\C) \cap \mathcal{C}^0(\C)$ by
$$ u_m(z) \coloneqq 
       \left\{
        \begin{array}{c@{,\quad}l}
         \max\Big(\abs{z}^2 + \varepsilon_j\chi\Big(\textstyle\frac{\abs{z-a_j}}{r_j}\Big)\log\abs{z-a_j},\, 1\Big) & \text{if }z \in \bigcup_{j=1}^m \D(a_j,r_j) \\
         \abs{z}^2 & \text{if }z \notin \bigcup_{j=1}^m \D(a_j,r_j)
        \end{array} 
       \right..
$$
Note that the sequence $(u_m)$ is monotonically decreasing. Thus $u \coloneqq \lim_{m \to \infty} u_m$ is a well-defined subharmonic function on $\C$. Moreover, since $\abs{z}^2 \ge u_m(z) \ge 1$ on $\C \setminus \D$ for every $m \in \N$, it is easy to observe that $u$ is continuous. By construction, $u(a_j) = 1$ for every $j \in \N$. \medskip

\noindent (2) If $v(a_j) = 1$ for every $j \in \N$ such that $a_j \in U$, then, by the choice of $(a_j)$ and continuity of $v$, it follows that $v \equiv 1$ on $b\D$. Since $v$ is subharmonic, we conclude that $v \le 1$ on $\D$.

Assume, in order to get a contradiction, that there exists a point $z_0 \in  \D$ such that $v(z_0) < 1$. The set $M \coloneqq \{z \in \D : v(z) = 1\}$ is closed by continuity of $v$, and it is open, since $v$ is subharmonic. Thus $M = \emptyset$, since $z_0 \in \D \setminus M$, and hence $v < 1$ on $\D$. 
By the classical Hopf Lemma, see \cite{Hopf52}, it then follows that 
$$ 
v(z) \le 1 - c \DIST(z,b\D) \quad \mathrm{for}\ z \in \bar\D
$$ 
for some constant $c > 0$. Since $v$ is $\mathcal{C}^1$-smooth, this implies that 
$$ 
v(z) > 1 \quad \mathrm{for}\ z \in \C \setminus \bar\D \mathrm{\ close\ enough\ to}\ b\D. 
$$
This contradicts the fact that $\{a_j\}_{j=1}^\infty \subset \C \setminus \bar\D$, $\lim_{j \to \infty}\abs{a_j} = 1$ and $v(a_j) = 1$ whenever $a_j \in U$. Thus $v \equiv 1$ on $\bar\D$. 
\end{proof}

In what follows, let $u \colon \Omega \to \R$ be the continuous subharmonic function constructed in the proof of  \Cref{thm:shdimension1}. Observe that not only $u(a_j) = 1$ for every $j \in \N$, but in fact for every $j \in \N$ there exists a disc $D_j \subset\subset \C \setminus \bar\D$ centered at $a_j$ such that $u \equiv 1$ on $D_j$. Define
$$ 
\sigma(z) \coloneqq \sum_{j=1}^\infty \delta_j \log\abs{z-a_j}, 
$$
where $(\delta_j)_{j=1}^\infty$ is a sequence of constants $\delta_j > 0$ converging to zero so fast that 
\begin{enumerate}
 \item[(1)] $\sigma \colon \C \to [-\infty, \infty)$ is subharmonic,
 \item[(2)] $\sigma < \frac{1}{4}$ on $\bar{\D}$,
 \item[(3)] $\sigma \ge -1$ on $\C \setminus \bigcup_{j=1}^\infty D_j$.
\end{enumerate} 
Further, fix a point $w_0 \in \C^{n-1}$ such that $\abs{w_0} = 4$, and let
$$ 
\Omega \coloneqq \{ (z,w) \in \C \times \C^{n-1} : \sigma(z) + \log\abs{w-w_0} + \abs{z}^2 + \abs{w}^2 < 3\}. 
$$
Observe that, after possibly passing to a suitable connected component, $\Omega$ is a pseudoconvex domain containing the set
$$ 
 E \coloneqq \big(\{a_j\}_{j=1}^\infty \times \C^{n-1}\big) \cup (\C \times \{w_0\}) 
$$
such that
\begin{enumerate}
 \item[(4)] $\Omega \cap \{\abs{w} \le 3\}$ is bounded,
 \item[(5)] $\Omega \cap \{\abs{w} \le 1\} \supset \bar\D \times \bar{\mathbb{B}}$,
 \item[(6)] $\Omega \cap \{2 \le \abs{w} \le 3\} \subset \big(\bigcup_{j=1}^\infty D_j\big) \times \C^{n-1}$.
\end{enumerate}
In order to see (4), note that $\sigma(z) > 0$ for $\abs{z} > \frac{5}{2}$, and $\log\abs{w-w_0} \ge 0$ for $\abs{w} \le 3$. For (5), observe that $\sigma < \frac{1}{4}$ on $\bar{\D}$ by (2), and $\log\abs{w-w_0} \le \log 5 < \frac{3}{4}$ for $w \in \bar{\B}$. Finally, for (6) note that $\log\abs{w-w_0} + \abs{z}^2 + \abs{w}^2 \ge 4$ for $2 \le \abs{w} \le 3$, so the claim follows from (3).

Let now $\psi \in \mathcal{PSH}(\Omega) \cap \mathcal{C}^1(\Omega)$ such that $\psi$ is bounded from above on $\Omega$.
From Liouville's theorem, it follows that $\psi \equiv c$ on $E$ for some $c \in \R$. Fix $w \in \B$. Then $\psi(a_j,w) = c$ for every $j \in \N$, and, since $\psi$ is $\mathcal{C}^1$-smooth and $\bar{\D} \times \{w\} \subset \Omega$ by (5), it thus follows from part (2) of \Cref{thm:shdimension1} that $\psi(\,\cdot\,,w) \equiv c$ on $\D \times \{w\}$. Since $w \in \B$ was arbitrary, this shows that $\psi \equiv c$ on $\omega \coloneqq \D \times \B$. 

It remains to show that there exists a function $\varphi \in \mathcal{PSH}(\Omega) \cap \mathcal{C}^0(\Omega)$ such that $\varphi$ is bounded from above on $\Omega$ and $\varphi$ is strictly plurisubharmonic on $\omega$. To that end, define $\tilde\varphi \colon \Omega \to \R$ by
$$ 
\tilde{\varphi}(z,w) \coloneqq \left\{\begin{array}{c@{,\quad}l} u(z) & \abs{w} < \frac{5}{2} \\ 1 & \abs{w} \ge \frac{5}{2} \end{array}\right..
$$
By the choice of the discs $D_j$ in the definition of the function $u$ above, and due to property (6), we see that the function $\tilde\varphi$ is continuous and plurisubharmonic on $\Omega$. Morevoer, by (4), it is bounded from above, and, by the choice of $u$, we have $\tilde{\varphi}(z,w) = \abs{z}^2$ on $\omega$. Let $\lambda \colon [0,\infty) \to [0,1]$ be smooth such that $\{\lambda > 0\} = [0,1)$ and $\lambda \equiv 1$ near $0$, and define $\theta \colon \Omega \to [0,\infty)$ by 
$$ 
\theta(z,w) \coloneqq \left\{\begin{array}{c@{,\quad}l} \lambda(\abs{z}^2)\abs{w}^2 & \abs{w} < \frac{5}{2} \\ 0 & \abs{w} \ge \frac{5}{2} \end{array}\right..
$$
Observe that $\theta$ is smooth and bounded from above on $\Omega$; here, the smoothness follows from the fact that $(\bar{\D} \times \{\abs{w} = \frac{5}{2}\}) \cap \Omega = \varnothing$ by (6). Thus, for every $c > 0$, the function
$$ 
\varphi \coloneqq \tilde{\varphi} + c\theta
$$
is continuous and bounded on $\Omega$. Moreover, if $c$ is chosen small enough, then it follows from the next Lemma (in the case $C \coloneqq \frac{1}{c}$) that $\varphi$ is plurisubharmonic on $\Omega$ and strictly plurisubharmonic on $\omega$. This concludes the proof of the theorem.

\begin{lem}
Let $\lambda \colon [0,\infty) \to [0,1]$ be smooth such that $\{\lambda > 0\} = [0,1)$ and $\lambda \equiv 1$ near $0$. Then for every $R>0$ there exists $C>0$ such that
$$ 
S(z_1,z') \coloneqq \lambda(\abs{z_1}^2)\abs{z'}^2 + C\abs{z_1}^2 
$$
is plurisubharmonic on $\C \times \B(0,R)$ and strictly plurisubharmonic on $\D \times \B(0,R)$; here, $z = (z_1, z') \in \C \times \C^{n-1}$. 
\end{lem}
\begin{proof}
For $z \in \C^n$ and $\xi \in \C^n$ let 
$$
H_S(z,\xi) \coloneqq \sum_{j,k=1}^\infty \frac{\partial^2 S}{\partial z_j \partial \zbar_k}(z)\xi_j\bar{\xi}_k.
$$
Then
\begin{equation*}\begin{split}
 H_S(z,\xi) = 
 &\big\{[\lambda''(\abs{z_1}^2)\abs{z_1}^2 + \lambda'(\abs{z_1}^2)] \cdot \abs{z'}^2 + C\big\} \cdot \abs{\xi_1}^2 \\
 &+ 2\RE\{\lambda'(\abs{z_1}^2)\bar{z}_1\xi_1 \langle z', \xi' \rangle\} + \lambda(\abs{z_1}^2) \cdot \abs{\xi'}^2, 
\end{split}\end{equation*}
where $\langle z', \xi' \rangle \coloneqq \sum_{j=2}^n z_j\bar{\xi}_j$. Let $R>0$ be fixed and let $z \in \C \times \B(0,R)$. Then
$$ 
\gabs{\lambda'(\abs{z_1}^2) \bar{z}_1\xi_1 \langle z', \xi' \rangle} \le R\abs{\lambda'(\abs{z_1}^2)} \cdot \abs{\xi_1}\abs{\xi'},
$$
and thus for $C > 0$ large enough it follows that
$$
H_S(z,\xi) \ge \frac{C}{2}\cdot\abs{\xi_1}^2 - R\abs{\lambda'(\abs{z_1}^2)} \cdot \abs{\xi_1}\abs{\xi'} + \lambda(\abs{z_1}^2) \cdot \abs{\xi'}^2.
$$
But an easy application of L'Hospital's rule shows that $(\lambda')^2 \le L\lambda$ for some constant $L > 0$, which implies that
$$
0 \le \frac{1}{2}\left(R\sqrt{L}\abs{\xi_1} - \sqrt{\lambda(\abs{z_1}^2)}\abs{\xi'}\right)^2 \le \frac{R^2L}{2} \cdot \abs{\xi_1}^2 - R \abs{\lambda'(\abs{z_1}^2)} \cdot \abs{\xi_1}\abs{\xi'} + \frac{\lambda(\abs{z_1}^2)}{2} \cdot \abs{\xi'}^2.
$$
Hence $H_S(z,\xi) \ge \varepsilon (\abs{\xi_1}^2 + \lambda(\abs{z_1}^2) \abs{\xi'}^2)$ for some constant $\varepsilon > 0$, provided that $C$ is chosen large enough.
\end{proof}

\section{Remark}
Recall that, by Theorem 5.5 in \cite{FornaessNarasimhan80}, for every Stein space $X$ and every plurisubharmonic function $\varphi$ on $X$ there exists a sequence of $\mathcal{C}^\infty$-smooth plurisubharmonic functions $\{\varphi_j\}_{j=1}^\infty$ such that $\lim_{j\to\infty} \varphi_j = \varphi$ (in fact, in the same theorem, it is shown that one can additionally guarantee that the sequence $(\varphi_j)$ consists of strictly plurisubharmonic functions and is monotonically decreasing). However, \Cref{thm:locboundedC0} and \Cref{thm:C0C1} show that in general it is not possible to approximate bounded plurisubharmonic functions by continuous bounded (from above) plurisubharmonic functions, or bounded continuous plurisubharmonic functions by $\mathcal{C}^1$-smooth bounded (from above) plurisubharmonic functions, respectively, even on pseudoconvex domains $\Omega \subset \C^n$ and even if the convergence is only required to be pointwise. (Note that the functions $\varphi$ constructed in the proofs of \Cref{thm:locboundedC0} and \Cref{thm:C0C1}, respectively, are bounded also from below.)

\end{document}

Z. S lodkowski, Pseudoconcave decompositions in complex manifolds, Advances in
complex geometry, 239 - 259, Contemp. Math., 735, Amer. Math. Soc., Providence,
RI, 2019.